\documentclass[11pt]{amsart}
\usepackage{amssymb,amsmath}
\usepackage{eucal}
\usepackage{epstopdf}
\usepackage{mathrsfs}
\usepackage{enumerate}
\usepackage{fullpage} 
\usepackage{setspace}
\usepackage{xcolor}
\usepackage{dsfont}

\usepackage{todonotes}
\usepackage{hyperref}

\newcommand{\mi}{\mathrm{i}}
\newcommand{\e}{\varepsilon}

\newcommand{\I}{\mathcal{I}}

\newcommand{\oh}{\mbox{$\frac{1}{2}$}}

\DeclareFontFamily{OT1}{rsfs}{}
\DeclareFontShape{OT1}{rsfs}{n}{it}{<-> rsfs10}{}
\DeclareMathAlphabet{\mathscr}{OT1}{rsfs}{n}{it}

\DeclareMathOperator{\meas}{meas}
\newtheorem{prop}{Proposition}[section]
\newtheorem{thm}[prop]{Theorem}
\newtheorem{theorem}[prop]{Theorem}

\newtheorem{lemma}[prop]{Lemma}

\newtheorem*{defn*}{Definition}
\newtheorem{conj}{Conjecture}

\numberwithin{equation}{section}

\renewcommand{\Re}{{\mathfrak{Re}}}

\renewcommand{\imath}{i}

\begin{document}

\title[]{Shifted moments of the Riemann zeta function}

\author{Nathan Ng}

\address{Department of Mathematics and
Computer Science, University of Lethbridge, Lethbridge, AB Canada T1K 3M4}
\email{nathan.ng@uleth.ca}

\author{Quanli Shen}

\address{Department of Mathematics and
Computer Science, University of Lethbridge, Lethbridge, AB Canada T1K 3M4}
\email{qlshen@outlook.com}

\author{Peng-Jie Wong}

\address{National Center for Theoretical Sciences\\
No. 1, Sec. 4, Roosevelt Rd., Taipei City, Taiwan}
\email{pengjie.wong@ncts.tw}



\date{\today}


\begin{abstract}
In this article, we prove that the Riemann hypothesis implies a conjecture of Chandee on shifted moments of the Riemann zeta function. The proof is based on ideas of Harper concerning sharp upper bounds for  the $2k$-th moments of the Riemann zeta function on the 
critical line.
\end{abstract}

\maketitle

\section{Introduction}

This article concerns the shifted moments of the Riemann zeta function 
\begin{equation}
  \label{IkTalphaj}
  I_k( T, \alpha_1 , \alpha_2) =  \int_{0}^{T} |\zeta( \oh + \mi (t+ \alpha_1)) |^{k} |\zeta( \oh + \mi (t+\alpha_2 )) |^{k} dt,
\end{equation}
where $T \ge 1$ and $\alpha_1:=\alpha_1(T), \alpha_2 :=\alpha_2(T)$ are real-valued functions satisfying 
\begin{equation}
  \label{alphaibd}
  |\alpha_1|, |\alpha_2| \le 0.5T.
\end{equation}
These are generalisations of the $2k$-th moments of the Riemann zeta function
$$
I_k(T) = \int_{0}^{T} |\zeta( \oh +\mi t)|^{2k} dt,
$$
since $ I_k(T)  =  I_k( T, 0,0 )$.  The theory of the moments of the Riemann zeta function is an important topic in  analytic number theory (see the classic books \cite{Ti}, \cite{Iv}, \cite{Mo}, and \cite{Ram95}). Unconditionally, Heap-Soundararajan \cite{HS} (for $0<k<1$) and {Radziwi\l\l}-Soundararajan \cite{RaSo} (for $k\ge 1$) proved that 
\begin{equation*}
I_k(T) \gg   T (\log T)^{k^2} .
\end{equation*}
Assuming the Riemann hypothesis, Harper \cite{Ha} showed that for any  $k \ge 0$,  
\begin{equation}
   \label{ub}
  I_k(T) \ll T(\log T)^{k^2}.
\end{equation}
Harper's argument builds on the work of Soundararajan \cite{So}, who showed that under the the Riemann hypothesis, for any $\e >0$, one has
\begin{align}
 I_k(T) \ll T(\log T)^{k^2+\e}.
 \label{soundbd}
\end{align}
Based on a random matrix model, Keating and Snaith \cite{KS} conjectured that for $k \in \mathbb{N}$, 
\begin{equation}
  \label{ksasymp}
I_k(T) \sim C_k T (\log T)^{k^2},
\end{equation}
for a precise constant $C_k$.
 By the classical works of Hardy-Littlewood \cite{HL} and Ingham \cite{In}, the asymptotic \eqref{ksasymp} is known, unconditionally, for $k=1,2$. Recently, 
 the first author \cite{Ng} showed that a certain conjecture for ternary additive divisor sums implies the validity of \eqref{ksasymp} for $k=3$.  In \cite{NSW}, the authors have shown that the Riemann hypothesis and a certain conjecture for quaternary additive divisor sums imply that \eqref{ksasymp} is true in the case $k=4$.  This work \cite{NSW} crucially uses the bounds for the shifted moments of the zeta function established in Theorem \ref{refined_Harper} below.

In \cite{Ch}, the more general shifted moments
\begin{equation*}
 \label{shiftedzeta}
M_{{\bf k}}(T,  {\bf \alpha}) = \int_{0}^{T} |\zeta(\oh + \mi (t+\alpha_1) )|^{2 k_1} \cdots |\zeta( \oh + \mi (t+\alpha_m)|^{2 k_m}  dt,
\end{equation*}
where ${\bf k} = (k_1, \ldots, k_m) \in (\mathbb{R}_{> 0})^m$ and ${\bf \alpha}= (\alpha_1, \ldots, \alpha_m) \in \mathbb{R}^m$,  were introduced. 
Chandee \cite[Theorems 1.1 and 1.2]{Ch} proved the following upper and lower bounds for $  I_{{\bf k}}(T,  {\bf \alpha})$.

\begin{thm}[Chandee]\label{Chandee-thm}
Let $k_i$ be positive real numbers. Let $\alpha_i=\alpha_i(T)$ be real-valued functions of $T$ such that $\alpha_i =o(T)$. Assume that $\lim_{T\rightarrow \infty} \alpha_i\log T$ and $\lim_{T\rightarrow \infty} (\alpha_i-\alpha_j)\log T$ exist or equal $\pm \infty$. Assume that for $i\neq j$, $\alpha_i\neq \alpha_j$ and  $\alpha_i-\alpha_j =O(1)$. 
Then the Riemann Hypothesis implies that for $T$ sufficiently large, one has
$$
  M_{{\bf k}}(T,  {\bf \alpha})\ll_{{\bf k},\varepsilon} T(\log T)^{k_1^2 + \cdots k_m^2 +\varepsilon} \prod_{i<j} \left( \min\left\{\frac{1}{|\alpha_i-\alpha_j|}, \log T \right\}\right)^{2k_ik_j}.
$$
Furthermore, if $k_i$ are positive integers, then for $T$ sufficiently large, unconditionally, one has
$$
  M_{{\bf k}}(T,  {\bf \alpha})\gg_{{\bf k}, \beta} T(\log T)^{k_1^2 + \cdots k_m^2 } \prod_{i<j}  \left( \min\left\{\frac{1}{|\alpha_i-\alpha_j|}, \log T \right\}\right)^{2k_ik_j},
$$
where
$$
\beta= \max_{ \{(i,j)\mid |\alpha_i-\alpha_j| =O(1/ \log (T))\} }\left\{\lim_{T\rightarrow \infty}  |\alpha_i-\alpha_j|\log T \right\}.
$$
\end{thm}

For the upper bound, Chandee used the techniques of Soundararajan \cite{So}; for the lower bound, Chandee's argument is based on the work of  Rudnick and Soundararajan \cite{RuSo}. 


%

Based on Keating and Snaith's random matrix model \cite{KS}, Chandee \cite[Conjecture 1]{Ch} made the following conjecture on shifted moments that  generalised   a conjecture of K\"{o}sters \cite{Ko} as follows:
\begin{conj}[Chandee]  \label{Chandee} 
Let $k \in \mathbb{N}$ and let ${\bf \alpha}= (\alpha_1, \alpha_2)$ be as in Theorem \ref{Chandee-thm}. Then one has
\begin{equation*}
  I_k( T, \alpha_1 , \alpha_2) \begin{cases}
 \asymp_k  T (\log T)^{k^2} & \text{ if }   \lim_{T \to \infty} |\alpha_1- \alpha_2| \log T = 0, \\
 \asymp_{k,c} T (\log T)^{k^2} & \text{ if }  \lim_{T \to \infty} |\alpha_1- \alpha_2| \log T = c \ne 0, \\
 \asymp_k T  \left(\frac{ \log T}{|\alpha_1- \alpha_2|}\right)^{\frac{k^2}{2}} & \text{ if } \lim_{T \to \infty} |\alpha_1- \alpha_2| \log T = \infty.
   \end{cases}
\end{equation*}
\end{conj}

Note that for any positive real $k$, $M_{{\bf k}}(T,  {\bf \alpha}) = I_{2k}( T, \alpha_1 , \alpha_2)$ for ${\bf k} =(k,k)$  and ${\bf \alpha}= (\alpha_1, \alpha_2)$. Therefore, Theorem \ref{Chandee-thm} of Chandee has established the conjectured  lower bound for $I_k( T, \alpha_1 , \alpha_2)$. It remains to prove the sharp upper bound for $I_k( T, \alpha_1 , \alpha_2)$ in order to establish Conjecture \ref{Chandee-thm}. In this article, assuming the Riemann hypothesis, we establish Chandee's conjecture by proving the following theorem.
\begin{theorem}\label{refined_Harper}
Let $k\ge 1$ be real. Let $\alpha_1$ and  $\alpha_2$ be real-valued functions $\alpha_i=\alpha_i(T)$ of $T$ which satisfy the bound \eqref{alphaibd}
and 
\begin{equation}
 \label{a1plusa2cond}
|\alpha_1+\alpha_2|\le T^{0.6}. 
\end{equation}
Then the Riemann hypothesis implies that for $T$ sufficiently large, we have
$$
 I_k( T, \alpha_1 , \alpha_2) \ll_{k} T(\log T)^{\frac{k^2}{2}}\mathcal{F}(T,\alpha_1,\alpha_2)^{\frac{k^2}{2}},
$$
where  $\mathcal{F}(T,\alpha_1,\alpha_2)$ is defined by
\begin{equation}\label{def-mathcalF}
\mathcal{F}(T, \alpha_1,\alpha_2) :=
\begin{cases}
 \min\left\{\frac{1}{|\alpha_1-\alpha_2|}, \log T \right\}  &  \text{ if $|\alpha_1-\alpha_2|\le \frac{1}{100}$;} \\
\log(2+ |\alpha_1-\alpha_2|) &  \text{ if $|\alpha_1-\alpha_2| >  \frac{1}{100}$.}
\end{cases}
\end{equation}
\end{theorem}
We establish this result by following the breakthrough work of Harper \cite{Ha}. 

%
%
%
\noindent {\bf Remarks}. 
\begin{enumerate}
\item  Theorem \ref{refined_Harper} contains Harper's result \eqref{ub} as a special case.  A key point is that Harper's method can be modified
so that the argument of \cite{Ha} still works 
when the shifts $\alpha_1, \alpha_2$ are introduced in \eqref{IkTalphaj}. One reason the argument works is that we are able to make use
of the  trigonometric identity
\begin{equation}
  \label{trigidentity}
  \text{ for } \theta_1, \theta_2 \in \mathbb{R},   \cos(\theta_1) + \cos(\theta_2) = 2 \cos \Big( \frac{\theta_1+\theta_2}{2} \Big) \cos \Big( \frac{\theta_1-\theta_2}{2} \Big)
\end{equation}
in \eqref{keyidentity} below. 
\item It  is natural to ask whether  Theorem \ref{refined_Harper}  can be extended to the more general moments 
$M_{{\bf k}}(T,  {\bf \alpha})$ where the components of ${\bf k} = (k_1, \ldots, k_m)$ are not necessarily equal and $m \ge 2$. 
\item In this theorem and throughout this article, whenever we write  ``sufficiently large $T$",   we mean that there exists $T_0:= T_0(k)$  a  positive parameter depending on $k$
such that $T \ge T_0$. 
\end{enumerate}


\noindent {\bf Conventions and notation}.
In this article, given two functions $f(x)$ and $g(x)$, we shall interchangeably use the notation  $f(x)=O(g(x))$, $f(x) \ll g(x)$, and $g(x) \gg f(x)$  to mean that there is $M >0$ such that $|f(x)| \le M |g(x)|$ for  sufficiently large $x$.  Given fixed parameters $\ell_1, \ldots, \ell_r$, the notation $f(x) \ll_{\ell_1, \ldots, \ell_r} g(x)$ means that the  $|f(x)| \le M g(x)$ where 
$M=M(\ell_1, \ldots, \ell_r)$ depends on the parameters 
$\ell_1, \ldots, \ell_r$.  The letter $p$ will always denote a prime number. In addition, $p_i$,  $p_{i}'$, and $q_i$ with $i \in \mathbb{N}$ shall denote prime numbers.

\section{Some tools}

We shall require the following tools, which are fundamental for the argument. Firstly, by a minor modification of the main Proposition of \cite{So} (see also \cite[Proposition 1]{Ha}), we have the following proposition providing an upper bound for the Riemann zeta function.

\begin{prop}\label{main-Sound}
Assume that the Riemann hypothesis holds. Let $\lambda_0=0.491\cdots$ denote the unique positive solution of $e^{-\lambda_0}=\lambda_0 + \lambda_0^2/2$.  Let $T$ be large. Then  for $\lambda\ge \lambda_0$, $2\le x\le T^2$ and $t\in [c_1T, c_2T]$ where $0 < c_1 < c_2$, one has
$$
\log|\zeta(\oh +\mi t)|
\le \Re\left(\sum_{ p\le x}\frac{1}{p^{\frac{1}{2}+\frac{\lambda}{ \log x}+\mi t } }\frac{\log(x/p)}{\log x } + \sum_{p\le\min\{\sqrt{x},\log T\} }\frac{1}{2p^{1+2 \mi t}} \right) + \frac{(1+\lambda)}{2}\frac{\log T}{\log x} +O(1).
$$
\end{prop}
 Also, we have the following variant of \cite[Lemma 4]{Ra}, which Harper formulates in  \cite[Proposition 2]{Ha}.

\begin{lemma}\label{formula-f}
Let $n=p_1^{a_1}\cdots p_r^{a_r}$, where $p_i$ are distinct primes, and $a_i\in \Bbb{N}$. Then for $T$ large, one has
$$
\int_T^{2T} \prod_{i=1}^r (\cos(t\log p_i))^{a_i} dt =Tg(n) +O(n),
$$ 
where the implied constant is absolute, and
$$
g(n)=\prod_{i=1}^r \frac{1}{2^{a_i}} \frac{a_i !}{((a_i/2)!)^2}
$$
if every $a_i$ is even, and $g(n)=0$ otherwise. Consequently, for $T$ large and any real number $\gamma$, we have
$$
\int_T^{2T} \prod_{i=1}^r (\cos( (t+\gamma)\log p_i))^{a_i} dt =(T+\gamma)g(n) +O(|\gamma|) + O(n),
$$ 
where the implied constants are absolute.
\end{lemma}

We shall also need the following further variant of   \cite[Lemma 4]{Ra} of {Radziwi\l\l}.

\begin{lemma}\label{refined-H-R}
Let $n=p_1^{a_1}\cdots p_r^{a_r} p_{r+1}^{a_{r+1}}\cdots p_{s}^{a_{s}}  $, where $p_i$ are distinct primes, and $a_i\in \Bbb{N}$.  Then we have
$$
\int_T^{2T} \prod_{ 1\le i\le r}(\cos(t\log p_i))^{a_i} \prod_{ r+1\le i\le s}(\cos(2t\log p_i))^{a_i} dt =Tg(n) +  O(( p_1^{a_1}\cdots p_r^{a_r})\cdot ( p_{r}^{2a_{r+1} }\cdots p_{s}^{2a_{s}})),
$$
where the implied constant is absolute.
Consequently, for any real $\gamma$, we have
\begin{align*}
&\int_T^{2T} \prod_{i=1}^r (\cos( (t+\gamma)\log p_i))^{a_i} \prod_{ r+1\le i\le s}(\cos(2(t+\gamma)\log p_i))^{a_i}  dt \\
&=(T+\gamma)g(n) +O(|\gamma|) + O( (p_1^{a_1}\cdots p_r^{a_r}) \cdot (p_{r+1}^{2a_{r+1}} \cdots p_{s}^{2a_{s}})),
\end{align*}
where the implied constants are absolute.
\end{lemma}

\begin{proof}
Following  {Radziwi\l\l} \cite[Proof of Lemma 4]{Ra}, for $c\in \Bbb{N}$, we can write
$$
(\cos(c t \log p_i))^{a_i} = \frac{1}{2^{a_i}}\left(e^{\mi ct\log p_i} + e^{-\mi ct\log p_i} \right)^{a_i} 
=  \frac{1}{2^{a_i}} {a_i \choose a_i/2}  +\sum_{a_i/2 \neq \ell_i \le a_i}  \frac{1}{2^{a_i}} {a_i \choose  \ell_i}  e^{\mi (a_i -2\ell_i) ct\log p_i},
$$
where ${a_i \choose a_i/2}=0$ if $a_i/2$ is not a positive integer.
Hence, setting $c_i=1$ for $1\le i\le r$ and  $c_i=2$ for $r+1\le i\le s$, we obtain
\begin{align*}
\prod_{ 1\le i\le s}(\cos(c_i t\log p_i))^{a_i}
&= \prod_{ 1\le i\le s}\left( \frac{1}{2^{a_i}} {a_i \choose  a_i/2}  +\sum_{a_i/2 \neq \ell_i \le a_i}  \frac{1}{2^{a_i}} {a_i \choose  \ell_i}  e^{\mi (a_i -2\ell_i) c_it\log p_i} \right)\\
&= g(n) + \sideset{}{'}\sum_{\ell_1,\ldots,\ell_s} \prod_{ 1\le i\le s}  \frac{1}{2^{a_i}}{a_i \choose  \ell_i}  e^{\mi (a_i -2\ell_i) c_it\log p_i}, 
\end{align*}
where the primed sum is over $(\ell_1,\ldots,\ell_s)\neq (\frac{a_1}{2},\ldots,\frac{a_s}{2})$ such that $1\le \ell_j\le a_j$ for every $1\le j\le s$. Thus, we deduce
\begin{equation}\label{mid-step}
\int_T^{2T} \prod_{ 1\le i\le r}(\cos(t\log p_i))^{a_i} \prod_{ r+1\le i\le s}(\cos(2t\log p_i))^{a_i} dt =Tg(n) + \sideset{}{'}\sum_{\ell_1,\ldots,\ell_s} \prod_{ 1\le i\le s}  \frac{1}{2^{a_i}}{a_i \choose  \ell_i}  \int_T^{2T} (*) dt.
\end{equation} 
The integrand $(*)$ is 
$$
 \exp\left(\mi t( b_1\log p_1 +\cdots + b_r\log p_r +  2b_{r+1}\log p_{r+1} +\cdots + 2b_s\log p_s \right )), 
$$
where $b_i =a_i -2\ell_i$. (Note that, as later, $b_1,\ldots,b_s$ cannot be all zero, and $|b_i|\le a_i$.) We then see  
$$
\left|\int_T^{2T} (*) dt\right|\le \frac{2}{|b_1\log p_1 +\cdots + b_r\log p_r +  2b_{r+1}\log p_{r+1} +\cdots + 2b_s\log p_s | }.
$$
(Note that the denominator is non-zero since $(b_1,\ldots,b_s)\neq(0,\ldots,0)$ and $p_1,\ldots, p_s$ are distinct.)
Grouping together those terms with $b_i > 0$ and $b_i < 0$, respectively, we
can write
$$
|b_1\log p_1 +\cdots + b_r\log p_r +  2b_{r+1}\log p_{r+1} +\cdots + 2b_s\log p_s | = |\log(M/N)|,
$$
where $M\neq N$ are positive integers. Without loss of generality, we may assume $M>N$ and obtain
$$
|\log(M/N)| = \log(M/N)\ge \log\left(\frac{N+1}{N} \right) =\log\left( 1+ \frac{1}{N} \right) \ge  \frac{1}{2N} \ge  \frac{1}{2  (p_1^{a_1}\cdots p_r^{a_r})( p_{r+1}^{2a_{r+1}}\cdots p_{s}^{2a_{s}})}.
$$
Therefore, the primed sum in \eqref{mid-step} is
$$
\ll  \sideset{}{'}\sum_{\substack{ 0\le \ell_i \le a_i\\ 1\le i\le s}} \prod_{ 1\le i\le s} 
 \frac{1}{2^{a_i}}{a_i \choose  \ell_i} \cdot  ( p_1^{a_1}\cdots p_r^{a_r} ) \cdot ( p_{r+1}^{2a_{r+1}}\cdots p_{s}^{2a_{s}}).
$$
Finally, observing that
$$
\sideset{}{'}\sum_{\ell_1,\ldots,\ell_s} \prod_{ 1\le i\le s}  \frac{1}{2^{a_i}}{a_i \choose  \ell_i} \le   \prod_{ 1\le i\le s}\sum_{ 0\le \ell_i \le a_i} \frac{1}{2^{a_i}}{a_i \choose  \ell_i} =   \prod_{ 1\le i\le s} \frac{1}{2^{a_i}} (1+1)^{a_i}=1,
$$
we complete the proof.
\end{proof}

Lastly, we recall the following variant of Mertens' estimate (see, e.g., \cite[p. 57]{GrSo} or \cite[Lemma 2.9]{Mu}).  

\begin{lemma}
Let $a$ and $z\ge 1$ be real numbers. Then one has
\begin{equation}\label{Mertens-variant}
\sum_{p\le z} \frac{\cos(a \log p)}{p}
\begin{cases}
= \log \left( \min \left\{\frac{1}{|a|}, \log z \right\} \right) +O(1)  &  \text{ if $|a|\le \frac{1}{100}$;} \\
\le \log\log(2+ |a|) +O(1)  &  \text{ if $|a|>\frac{1}{100}$,}
\end{cases}
\end{equation}
where the implied constants are absolute.
\end{lemma}

\section{Setup and outline of the proof of  Theorem \ref{refined_Harper}}

The goal of this section is to prove Theorem \ref{refined_Harper}. To do so, we follow closely Harper \cite{Ha}. We let $\beta_0 =0$ and 
$$
\beta_i = \frac{20^{i-1}}{(\log\log T)^2}
$$
for every integer $i\ge 1$. Define $\mathcal{I}= \mathcal{I}_{k,T} =1+\max\{i\mid \beta_i\le e^{-1000k}\}.$
For $1\le i\le j \le \mathcal{I}$, we set
\begin{equation*}
G_{i,j}(t)=G_{i,j,T,\alpha_1,\alpha_2}(t)= \sum_{T^{\beta_{i-1}}< p\le T^{\beta_{i}}}\frac{\cos(\frac{1}{2}(\alpha_1-\alpha_2)\log p)}{p^{\frac{1}{2}+\frac{1}{\beta_j \log T}+\mi (t+\frac{1}{2}(\alpha_1+\alpha_2)) } }\frac{\log(T^{\beta_j}/p)}{\log T^{\beta_j} }. 
\end{equation*}
For $1\le i\le \mathcal{I}$, we set 
\begin{equation}
 \label{Fi}
F_i(t) =G_{i,\mathcal{I}}(t)= \sum_{T^{\beta_{i-1}}< p\le T^{\beta_{i}}}\frac{\cos(\frac{1}{2}(\alpha_1-\alpha_2)\log p)}{p^{\frac{1}{2}+\frac{1}{\beta_{\I} \log T}+\mi (t+\frac{1}{2}(\alpha_1+\alpha_2)) } }\frac{\log(T^{\beta_{\I}}/p)}{\log T^{\beta_{\I}} }.
\end{equation}
We define 
 \begin{equation*}
 \begin{split}
  \mathcal{S}(0) = \mathcal{S}_{T,\alpha_1,\alpha_2}(0) & := \Big\{ t\in [T, 2T] \mid  |\Re G_{1,\ell}(t)|> \beta_{1}^{-3/4} \text{ for some } 1\le \ell \le \mathcal{I}\Big\}.
 \end{split}
 \end{equation*}
For $1\le j\le \mathcal{I}-1$, we define 
\begin{align*}
\mathcal{S}(j)
 & = \mathcal{S}_{k,T,\alpha_1,\alpha_2}(j)\\
 & := \Big\{ t\in [T, 2T] \mid  |\Re G_{i,\ell}(t)|\le \beta_i^{-3/4}  \text{ for every }  (i,\ell) \in \mathbb{N}^2 \text{ such that } 1\le i\le j \text{ and } i\le \ell \le  \I,   \\
& \hspace{2.9cm} \text{ and } |\Re G_{j+1,\ell'}(t)|> \beta_{j+1}^{-3/4} \text{ for some } j+1\le \ell' \le \mathcal{I}\Big\}.
\end{align*}
Finally, we set
\begin{equation}\label{def-T}
\mathcal{T}= \mathcal{T}_{k,T,\alpha_1,\alpha_2} := \Big\{ t\in [T, 2T] \mid | \Re  F_i(t) |\le \beta_i^{-3/4} \text{ for every } 1\le i\le \I \Big\}.
\end{equation}
Note that $\beta_{j+1}\le\beta_{ \mathcal{I}}  \le 20 e^{-1000k}$ for any $1\le j \le \mathcal{I}-1$.

Observe 
$$
[T,2T] = \bigcup_{j=0}^{\mathcal{I}-1} \mathcal{S}(j) \cup \mathcal{T}.
$$
Thus, in order to prove Theorem \ref{refined_Harper}, it is sufficient to prove 
\begin{equation}
\begin{split}
  \label{keyinequality}
& \sum_{j=0}^{\mathcal{I}-1} \int_{t\in \mathcal{S}(j) } |\zeta(\oh+\mi(t+\alpha_1))|^k|\zeta(\oh+\mi (t+\alpha_2))|^k dt
+  \int_{t\in\mathcal{T} } |\zeta(\oh+\mi(t+\alpha_1))|^k|\zeta(\oh +\mi (t+\alpha_2))|^k dt\\
&\ll T  (\log T)^{\frac{k^2}{2}} \mathcal{F}(T,\alpha_1, \alpha_2)^{\frac{k^2}{2}}.
\end{split}
\end{equation}
Applying Proposition \ref{main-Sound} with $\lambda=1$, for sufficiently large $T$, $2\le x\le T^2$, and $t\in [T,2T]$, we have
\begin{align}\label{upper-log-zeta}
 \begin{split}
&\log|\zeta(\oh+\mi (t +\alpha_i))| \\
&\le \Re\left(\sum_{ p\le x}\frac{1}{p^{\frac{1}{2}+\frac{1}{ \log x}+\mi (t +\alpha_i)} }\frac{\log(x/p)}{\log x } + \sum_{p\le\min\{\sqrt{x},\log T\} }\frac{1}{2p^{1+2 \mi (t +\alpha_i)}} \right) +\frac{\log T}{\log x} +O(1).
 \end{split}
\end{align}
We further note that the ``main term'' for the upper bound of $\log(  |\zeta(\frac{1}{2}+\mi(t+\alpha_1))|^k|\zeta(\frac{1}{2}+\mi (t+\alpha_2))|^k )$ derived from \eqref{upper-log-zeta} is
\begin{equation}
\begin{split}
 \label{keyidentity}
&k\Re \sum_{ p\le x}\frac{1}{p^{\frac{1}{2}+\frac{1}{ \log x}+\mi (t+\alpha_1) } }\frac{\log(x/p)}{\log x } +k\Re\sum_{ p\le x}\frac{1}{p^{\frac{1}{2}+\frac{1}{ \log x}+\mi (t+\alpha_2) } }\frac{\log(x/p)}{\log x } \\
&=k \sum_{ p\le x}\frac{\cos(-(t+\alpha_1)\log p)}{p^{\frac{1}{2}+\frac{1}{ \log x}} }\frac{\log(x/p)}{\log x } +k \sum_{ p\le x}\frac{\cos(-(t+\alpha_2)\log p)}{p^{\frac{1}{2}+\frac{1}{ \log x} } }\frac{\log(x/p)}{\log x }\\
&=k \sum_{ p\le x}\frac{1}{p^{\frac{1}{2}+\frac{1}{ \log x} } }\frac{\log(x/p)}{\log x } (2\cos(-(t+\oh(\alpha_1+\alpha_2))\log p) \cos(-\oh(\alpha_1-\alpha_2)\log p)  )\\
&=2k \Re \sum_{ p\le x}\frac{\cos(\frac{1}{2}(\alpha_1-\alpha_2)\log p)}{p^{\frac{1}{2}+\frac{1}{ \log x}+\mi(t+\frac{1}{2}(\alpha_1+\alpha_2))  } }\frac{\log(x/p)}{\log x }, 
\end{split}
\end{equation}
where we have made use of the trigonometric identity \eqref{trigidentity}. 
Arguing similarly for the second sum in \eqref{upper-log-zeta}, we arrive at
\begin{equation}
\begin{split}
\label{upper-log-zeta-k}
 & \log(  |\zeta(\oh+\mi(t+\alpha_1))|^k|\zeta(\oh+\mi (t+\alpha_2))|^k ) \\
 & \le
2k \Re\left(\sum_{ p\le x}\frac{\cos(\frac{1}{2}(\alpha_1-\alpha_2)\log p)}{p^{\frac{1}{2}+\frac{1}{ \log x}+\mi (t+\frac{1}{2}(\alpha_1+\alpha_2)) } }\frac{\log(x/p)}{\log x } + \sum_{p\le\min\{\sqrt{x},\log T\} }\frac{\cos((\alpha_1-\alpha_2)\log p)}{2p^{1+ \mi (2t+(\alpha_1+\alpha_2)) }} \right) +\frac{2k\log T}{\log x}\\
& +O(k).
\end{split}
\end{equation}

Theorem \ref{refined_Harper} will be deduced from the following three lemmata.

\begin{lemma}\label{refind_Harper_Lemma1}
In the notation and assumption as above and Theorem \ref{refined_Harper}, for any sufficiently large $T$,  we have
\begin{equation}\label{refind_Harper_est1}
\int_{t\in \mathcal{T}} \exp \left( 2k \Re\sum_{ p\le T^{\beta_{\I}}}\frac{\cos(\frac{1}{2}(\alpha_1-\alpha_2)\log p)}{p^{\frac{1}{2}+\frac{1}{\beta_{\I} \log T}+\mi (t+\frac{1}{2}(\alpha_1+\alpha_2)) } }\frac{\log(T^{\beta_{\I}}/p)}{\log T^{\beta_{\I}} }  \right)dt
\ll_k T (\log T)^{\frac{k^2}{2}} \left(\mathcal{F}(T,\alpha_1,\alpha_2) \right)^{\frac{k^2}{2}},
\end{equation}
where  $\mathcal{F}(T,\alpha_1,\alpha_2) $ is defined in \eqref{def-mathcalF}.
\end{lemma}

\begin{lemma}\label{refind_Harper_Lemma2}
In the notation and assumption as above, let $T$ be sufficiently large. Then we have 
 $$
 \meas(\mathcal{S}(0)) \ll_k Te^{-(\log\log T)^2/10}.
 $$
 In addition, for  $ 1 \le j \le  \I-1$, we have
\begin{align}\label{refind_Harper_est2}
 \begin{split}
&\int_{t\in \mathcal{S}(j)} \exp \left( 2k \Re\sum_{ p\le T^{\beta_{j}}}\frac{\cos(\frac{1}{2}(\alpha_1-\alpha_2)\log p)}{p^{\frac{1}{2}+\frac{1}{\beta_{j} \log T}+\mi (t+\frac{1}{2}(\alpha_1+\alpha_2)) } }\frac{\log(T^{\beta_{j}}/p)}{\log T^{\beta_{j}} }  \right)dt \\
&\ll_k T (\log T)^{\frac{k^2}{2}} \mathcal{F}(T,\alpha_1,\alpha_2)^{\frac{k^2}{2}} \exp \left( \frac{\log(1/\beta_j)}{21 \beta_{j+1}}\right).
 \end{split}
\end{align}
\end{lemma}

\begin{lemma}\label{refind_Harper_Lemma3}
In the notation and assumption as above and Theorem \ref{refined_Harper}, we have 
\begin{equation}
\begin{split}
\label{replace}
 & \int_{t\in \mathcal{T}} \exp \left( 2k \Re \left( 
 \sum_{ p\le T^{\beta_{\I}}}\frac{\cos(\frac{1}{2}(\alpha_1-\alpha_2)\log p)}{p^{\frac{1}{2}+\frac{1}{\beta_{\I} \log T}+\mi (t+\frac{1}{2}(\alpha_1+\alpha_2)) } }\frac{\log(T^{\beta_{\I}}/p)}{\log T^{\beta_{\I}} } + \sum_{p\le \log T }\frac{\cos((\alpha_1-\alpha_2)\log p)}{2p^{1+ \mi (t+(\alpha_1-\alpha_2))}}
\right)  \right) dt \\
& \ll_k T (\log T)^{\frac{k^2}{2}} \left(\mathcal{F}(T,\alpha_1,\alpha_2) \right)^{\frac{k^2}{2}}.
\end{split}
\end{equation}
and for $1\le j\le \mathcal{I}-1$
\begin{equation}
\begin{split}
\label{replace2}
 & \int_{t\in \mathcal{S}(j)} \exp \left( 2k \Re \left(
\sum_{ p\le T^{\beta_{j}}}\frac{\cos(\frac{1}{2}(\alpha_1-\alpha_2)\log p)}{p^{\frac{1}{2}+\frac{1}{\beta_{j} \log T}+\mi (t+\frac{1}{2}(\alpha_1+\alpha_2)) } }\frac{\log(T^{\beta_{j}}/p)}{\log T^{\beta_{j}} } + \sum_{p\le \log T }\frac{\cos((\alpha_1-\alpha_2)\log p)}{2p^{1+ \mi (t+(\alpha_1-\alpha_2))}}
\right)  \right) dt \\
& \ll_k T (\log T)^{\frac{k^2}{2}} \mathcal{F}(T,\alpha_1,\alpha_2)^{\frac{k^2}{2}} \exp \left( \frac{\log(1/\beta_j)}{21 \beta_{j+1}}\right).
\end{split}
\end{equation}
\end{lemma}

Now we are ready to prove Theorem \ref{refined_Harper}.

\begin{proof}[Proof of Theorem \ref{refined_Harper}]
We must show inequality  \eqref{keyinequality}  holds.  It suffices to show that each of the two terms on the left hand side of \eqref{keyinequality}
is  $\ll T  (\log T)^{\frac{k^2}{2}} \mathcal{F}(T,\alpha_1, \alpha_2)^{\frac{k^2}{2}}$. 
By \eqref{upper-log-zeta-k}, we know that $\log( |\zeta(\frac{1}{2}+\mi (t+\alpha_1))|^k|\zeta(\frac{1}{2}+\mi(t+\alpha_2))|^k )$ is at most
$$
2k \Re\left( \sum_{ p\le T^{\beta_{\I}}}\frac{\cos(\frac{1}{2}(\alpha_1-\alpha_2)\log p)}{p^{\frac{1}{2}+\frac{1}{\beta_{\I} \log T}+\mi (t+\frac{1}{2}(\alpha_1+\alpha_2)) } }\frac{\log(T^{\beta_{\I}}/p)}{\log T^{\beta_{\I}} }  + \sum_{p\le \log T }\frac{\cos((\alpha_1-\alpha_2)\log p)}{2p^{1+ \mi( 2t+(\alpha_1+\alpha_2))}} \right) + \frac{2k}{\beta_{\mathcal{I}}} +O(k).
$$
Hence, \eqref{replace} of  Lemma \ref{refind_Harper_Lemma3}  implies
\begin{equation}
\begin{split}
  \label{inequalityA}
& \int_{t\in\mathcal{T} } |\zeta(\oh+\mi (t+\alpha_1))|^k|\zeta(\oh+\mi(t+\alpha_2))|^k dt \\
&\ll  \int_{t\in \mathcal{T}} \exp\left( 2k  \Re\left(\sum_{ p\le T^{\beta_{\I}}}\frac{\cos(\frac{1}{2}(\alpha_1-\alpha_2)\log p)}{p^{\frac{1}{2}+\frac{1}{\beta_{\I} \log T}+\mi (t+\frac{1}{2}(\alpha_1+\alpha_2)) } }\frac{\log(T^{\beta_{\I}}/p)}{\log T^{\beta_{\I}} }  + \sum_{p\le \log T }\frac{\cos((\alpha_1-\alpha_2)\log p)}{2p^{1+ \mi(2 t+(\alpha_1+\alpha_2))}} \right)  \right) dt\\
&\times e^{2k/\beta_{\I} +O(k)} \\
&\ll_k   T  (\log T)^{\frac{k^2}{2}} \mathcal{F}(T,\alpha_1,\alpha_2)^{\frac{k^2}{2}} .
\end{split}
\end{equation}
Similarly, for $1\le j \le \mathcal{I}-1$, we can bound  $\log( |\zeta(\frac{1}{2}+\mi (t+\alpha_1))|^k|\zeta(\frac{1}{2}+\mi(t+\alpha_2))|^k )$ above by
$$
2k \Re\left(\sum_{ p\le T^{\beta_{j}}}\frac{\cos(\frac{1}{2}(\alpha_1-\alpha_2)\log p)}{p^{\frac{1}{2}+\frac{1}{\beta_{j} \log T}+\mi (t+\frac{1}{2}(\alpha_1+\alpha_2)) } }\frac{\log(T^{\beta_{j}}/p)}{\log T^{\beta_{j}} } + \sum_{p\le \log T }\frac{\cos((\alpha_1-\alpha_2)\log p)}{2p^{1+\mi(2 t+(\alpha_1+\alpha_2))}} \right) + \frac{2k}{\beta_{j}} +O(k).
$$
It then follows from Lemma \ref{refind_Harper_Lemma3} that
$$
 \int_{t\in\mathcal{S}(j) }|\zeta(\oh+\mi (t+\alpha_1))|^k|\zeta(\oh+\mi(t+\alpha_2))|^k  dt
 \ll e^{2k/\beta_j}\cdot e^{- (21 \beta_{j+1})^{-1} \log(1/\beta_{j+1})}  T  (\log T)^{\frac{k^2}{2}}\mathcal{F}(T,\alpha_1,\alpha_2)^{\frac{k^2}{2}}.
$$
Since $20\beta_j= \beta_{j+1}\le \beta_{\mathcal{I}}\le 20 e^{-1000k}$, $\log(1/\beta_{j+1})\ge 900k$, and so
$$
e^{2k/\beta_j}\cdot e^{-(21\beta_{j+1})^{-1} \log(1/\beta_{j+1})} = e^{2k/\beta_j -  (\log(1/\beta_{j+1}))/420 \beta_j} \le e^{-0.1k/\beta_j}.
$$
Observe that $\I \le \frac{2}{(\log 20)}\log\log\log T$ and
\begin{equation}
  \label{bd1}
\sum_{j=1}^{\I-1} e^{-0.1k/\beta_j} = \sum_{j=1}^{\I-1}  e^{-2k(\log\log T)^2/ 20^j} \le
 e^{-2k(\log\log T)^2}+
 \int_{1}^{\frac{2}{(\log 20)}\log\log\log T} e^{-2k(\log\log T)^2/ 20^x} dx.
\end{equation}
By the change of variables $20^{-x} =u$ (with $dx = \frac{-1}{\log 20} \frac{du}{u}  $), we see that the integral above equals
\begin{equation}
  \label{bd2}
-\frac{1}{\log 20}\int_{\frac{1}{20}}^{ \frac{1}{(\log\log T)^2}}    e^{-2k(\log\log T)^2 u}   \frac{du}{u}  \ll (\log\log T)^2  \int_{ \frac{1}{(\log\log T)^2}}^{\frac{1}{20}}  e^{-2k(\log\log T)^2 u}  du  \ll \frac{e^{-2k}}{2k}.
\end{equation}
Combining \eqref{bd1} and \eqref{bd2}, we arrive at
\begin{equation}
  \label{inequalityB}
\sum_{j=1}^{\mathcal{I}-1} \int_{t\in\mathcal{S}(j) } |\zeta(\oh+\mi (t+\alpha_1))|^k|\zeta(\oh+\mi(t+\alpha_2))|^k  dt \ll T  (\log T)^{\frac{k^2}{2}}\mathcal{F}(T,\alpha_1,\alpha_2)^{\frac{k^2}{2}}.
\end{equation}
Finally, for $j=0$, by the Cauchy-Schwarz inequality and Lemma \ref{refind_Harper_Lemma2}, we have
\begin{align}
\label{s0}
\begin{split}
&\int_{t\in\mathcal{S}(0) } |\zeta(\oh+\mi (t+\alpha_1))|^k|\zeta(\oh+\mi(t+\alpha_2))|^k  dt \\
&\le   \meas(\mathcal{S}(0))^{\frac{1}{2}} \left( \int_T^{2T} |\zeta(\oh+\mi (t+\alpha_1))|^{2k}|\zeta(\oh+\mi (t+\alpha_2))|^{2k} dt  \right)^{\frac{1}{2}}.
\end{split}
\end{align}
Using the Cauchy-Schwarz inequality again and the upper bound \eqref{soundbd} with $\e =1$,  we see
\begin{equation*}
\begin{split}
&\int_T^{2T} |\zeta(\oh+\mi (t+\alpha_1))|^{2k}|\zeta(\oh+\mi (t+\alpha_2))|^{2k} dt \\
&\ll  \left( \int_T^{2T} |\zeta(\oh+\mi (t+\alpha_1))|^{4 k} dt  \right)^{\frac{1}{2}}  
     \left(  \int_T^{2T} |\zeta(\oh+\mi (t+\alpha_2))|^{4k} dt  \right)^{\frac{1}{2}}\\
 & \ll T (\log T)^{4k^2 + 1}.
\end{split}
\end{equation*}
Hence, combined with \eqref{s0}, we derive
\begin{equation}
\begin{split}
 \label{inequalityC}
\int_{t\in\mathcal{S}(0) } |\zeta(\oh+\mi (t+\alpha_1))|^k|\zeta(\oh+\mi(t+\alpha_2))|^k  dt 
\ll \sqrt{T}e^{-(\log\log T)^2/20}\cdot \sqrt{T} (\log T)^{2k^2+\frac{1}{2}}\ll T.
\end{split}
\end{equation}
Therefore, by combining inequalities \eqref{keyinequality}, \eqref{inequalityA}, \eqref{inequalityB}, and \eqref{inequalityC} 
the proof of Theorem \ref{refined_Harper}  is complete. 
\end{proof}

\section{Proof of Lemma \ref{refind_Harper_Lemma1}}
Observe that 
\begin{equation}
   \label{decomp}
      \sum_{ p\le T^{\beta_{\I}}}\frac{\cos(\frac{1}{2}(\alpha_1-\alpha_2)\log p)}{p^{\frac{1}{2}+\frac{1}{\beta_{\I} \log T}+\mi (t+\frac{1}{2}(\alpha_1+\alpha_2)) } }\frac{\log(T^{\beta_{\I}}/p)}{\log T^{\beta_{\I}} }
      = \sum_{ i=1}^{\I} F_i(t)
\end{equation}
where $F_i$ is defined by \eqref{Fi}.
By \eqref{decomp},  we have 
\begin{align}\label{Ha-lemma1-start}
  \begin{split}
&\int_{t\in \mathcal{T}} \exp \left( 2k \Re\sum_{ p\le T^{\beta_{\I}}}\frac{\cos(\frac{1}{2}(\alpha_1-\alpha_2)\log p)}{p^{\frac{1}{2}+\frac{1}{\beta_{\I} \log T}+\mi (t+\frac{1}{2}(\alpha_1+\alpha_2)) } }\frac{\log(T^{\beta_{\I}}/p)}{\log T^{\beta_{\I}} }  \right)dt
= \int_{t\in \mathcal{T}} \prod_{1\le i \le \I} \exp( k\Re F_i(t) )^2 dt 
 \end{split}
\end{align}
where we recall $\mathcal{T}$ 
  is defined  in \eqref{def-T}. Next, note that 
\begin{equation}
 \label{Ha-lemma1-startB}
   \int_{t\in \mathcal{T}} \prod_{1\le i \le \I} \exp( k\Re F_i(t) )^2 dt \ll \mathscr{I}
 :=  \int_T^{2T}  \prod_{1\le i \le \I}\left(\sum_{0\le j \le 100k \beta_i^{-3/4}} \frac{(k \Re F_i(t)  )^j}{j!} \right)^2 dt.
\end{equation}
This inequality establishes  that each factor $\exp(k \Re F_i(t))$ may be replaced by its Taylor polynomial
of length $100k \beta_{i}^{-\frac{3}{4}}$.  Full details of this argument can be found in \cite[Lemma 5.2, pp. 484-486]{Ki}.
In order to simplify the presentation, we set 
\begin{equation*}
\gamma^+= \frac{1}{2}(\alpha_1+\alpha_2) \text{ and } \gamma^-= \frac{1}{2}(\alpha_1-\alpha_2).
\end{equation*} 
Expanding out all of the $j$-th powers and opening the square, we see  that 
\begin{align} \label{Ha-lemma1-mid1}
 \begin{split}
& \mathscr{I}\\ & = \sum_{\tilde{j}, \tilde{\ell}}\left(\prod_{1\le i \le \I} \frac{k^{j_i}}{j_i !}\frac{k^{\ell_i}}{\ell_i !}  \right) \sum_{\tilde{p}, \tilde{q}} C(\tilde{p}, \tilde{q}) \int_T^{2T}\prod_{1\le i \le \I} \prod_{\substack{1\le r \le j_i  \\ 1\le s \le \ell_i}} \cos((t+\gamma^+)\log p(i,r) )\cos((t+\gamma^+)\log q(i,s) ) dt\\
& \times \prod_{1\le i \le \I} \prod_{\substack{1\le r \le j_i  \\ 1\le s \le \ell_i}} \cos(\gamma^-\log p(i,r) )\cos(\gamma^-\log q(i,s) ), 
 \end{split}
\end{align}
where the first sum is over all 
\[
   \tilde{j}=(j_1,\ldots,j_{\mathcal{I}}), \, \tilde{\ell}=(\ell_1,\ldots,\ell_{\mathcal{I}}) \text{ where } 0\le j_i,\ell_i \le 100k \beta_i^{-3/4},
\]
the second sum is over
\begin{equation*}
\begin{split}
   \tilde{p} & =( p(1,1),\ldots, p(1,j_1), p(2,1),\ldots,p(2,j_2),\ldots, p(\mathcal{I},j_{\mathcal{I}} ))
   \text{ and } \\
    \tilde{q} & =( q(1,1),\ldots, q(1,\ell_1), q(2,1),\ldots,q(2,{\ell}_2),\ldots, q(\mathcal{I},{\ell}_{\mathcal{I}} ))
\end{split}
\end{equation*}
whose components are primes which satisfy
$$
T^{\beta_{i-1}}< p(i,1),\ldots, p(i,j_i),q(i,1),\ldots, q(i,\ell_i)\le T^{\beta_{i}}
$$
for any $1\le i\le \mathcal{I}$, and
$$
C(\tilde{p}, \tilde{q})= \prod_{1\le i \le \I}\prod_{\substack{1\le r \le j_i  \\ 1\le s \le \ell_i}}\frac{1}{p(i,r)^{\frac{1}{2}+\frac{1}{\beta_{\I} \log T} } }\frac{\log(T^{\beta_{\I}}/p(i,r))}{\log T^{\beta_{\I}} } 
\frac{1}{q(i,s)^{\frac{1}{2}+\frac{1}{\beta_{\I} \log T} } }\frac{\log(T^{\beta_{\I}}/q(i,s))}{\log T^{\beta_{\I}} } .
$$
Following the argument in  \cite[p. 10]{Ha} (see the third displayed equation), we have 
\begin{equation}
  \label{nbound}
   \prod_{1\le i \le \I}\prod_{\substack{1\le r \le j_i  \\ 1\le s \le \ell_i}} p(i,r)q(i,s) \le T^{0.1}. 
\end{equation} 
By Lemma \ref{formula-f} and \eqref{nbound}, it follows that
\begin{equation}
\begin{split}
  \label{lemmaapp}
& \int_T^{2T}\prod_{1\le i \le \I} \prod_{\substack{1\le r \le j_i  \\ 1\le s \le \ell_i}} \cos((t+\gamma^+)\log p(i,r) )\cos((t+\gamma^+)\log q(i,s) ) dt\\
 & = (T+\gamma^+) g\left( \prod_{1\le i \le \I} \prod_{\substack{1\le r \le j_i  \\ 1\le s \le \ell_i}}  p(i,r) q(i,s)  \right) + O(|\gamma^+|) +O(T^{0.1}).
\end{split}
\end{equation}
Let 
\begin{equation*}
   D(\tilde{p}, \tilde{q})=  \prod_{1\le i \le \I}  \prod_{\substack{1\le r \le j_i  \\ 1\le s \le \ell_i}} \frac{1}{\sqrt{p(i,r)}}\frac{1}{\sqrt{ q(i,s)  }} 
\end{equation*}
and observe that
\begin{equation}
 \label{CDinequality}
 C(\tilde{p}, \tilde{q})\le D(\tilde{p}, \tilde{q}).  
\end{equation} 
By \eqref{lemmaapp}, 
\eqref{CDinequality},  and the bound $|\cos x|\le 1$ for real $x$, it follows that \eqref{Ha-lemma1-mid1} equals
\begin{align} \label{Ha-lemma1-mid2}
 \begin{split}
 \mathscr{I}  & =(T+\gamma^+) \sum_{\tilde{j}, \tilde{\ell}}\left(\prod_{1\le i \le \I} \frac{k^{j_i}}{j_i !}\frac{k^{\ell_i}}{\ell_i !}  \right) \sum_{\tilde{p}, \tilde{q}} C(\tilde{p}, \tilde{q}) g\left( \prod_{1\le i \le \I} \prod_{\substack{1\le r \le j_i  \\ 1\le s \le \ell_i}}  p(i,r) q(i,s)  \right)\\
& \times \prod_{1\le i \le \I} \prod_{\substack{1\le r \le j_i  \\ 1\le s \le \ell_i}} \cos(\gamma^-\log p(i,r) )\cos(\gamma^-\log q(i,s) ) \\
&+O\left( (|\gamma^+|+T^{0.1}) \sum_{\tilde{j}, \tilde{\ell}}\left(\prod_{1\le i \le \I} \frac{k^{j_i}}{j_i !}\frac{k^{\ell_i}}{\ell_i !}  \right) \sum_{\tilde{p}, \tilde{q}} D(\tilde{p}, \tilde{q}) \right).
  \end{split}
\end{align}
By the argument of Harper \cite[p. 10]{Ha}, it can be shown that the big-$O$ term above is at most $(|\gamma^+|+T^{0.1}) T^{0.1}(\log\log T)^{2k}$. 

The  inner summand in \eqref{Ha-lemma1-mid2} is
$$
 C(\tilde{p}, \tilde{q}) g\left( \prod_{1\le i \le \I} \prod_{\substack{1\le r \le j_i  \\ 1\le s \le \ell_i}}  p(i,r) q(i,s)  \right)  \prod_{1\le i \le \I} \prod_{\substack{1\le r \le j_i  \\ 1\le s \le \ell_i}} \cos(\gamma^-\log p(i,r) )\cos(\gamma^-\log q(i,s) ).
$$
Since $g$ is supported on squares,  this expression is non-zero if and only if
$$\prod_{1\le i \le \I} \prod_{\substack{1\le r \le j_i  \\ 1\le s \le \ell_i}}  p(i,r) q(i,s)=p_{1}^2 \cdots p_{N}^2$$
for some $N \in \mathbb{N}$. 
In this case, we have 
\begin{equation}
  \label{nonnegativity}
 \prod_{1\le i \le \I} \prod_{\substack{1\le r \le j_i  \\ 1\le s \le \ell_i}} \cos(\gamma^-\log p(i,r) )\cos(\gamma^-\log q(i,s) ) 
 =\cos^2(\gamma^-\log p_1) \cdots \cos^2(\gamma^- \log p_N)
  \ge 0.
\end{equation}
By \eqref{Ha-lemma1-start}, 
\eqref{Ha-lemma1-mid1}, 
\eqref{Ha-lemma1-mid2}, \eqref{nonnegativity}, and \eqref{CDinequality}, we deduce that  
\begin{align*}
\mathscr{I} &\ll T \prod_{1\le i \le \I} \sum_{0\le j,\ell \le 100 \beta_i^{-3/4} }\frac{k^{j+\ell}}{j!\ell!} \sum_{T^{\beta_{i-1}}< p_1,\ldots,p_j,q_1,\ldots, q_{\ell} \le T^{\beta_{i}}} \frac{g(p_1\cdots p_jq_1\cdots q_{\ell})}{\sqrt{p_1\cdots p_jq_1\cdots q_{\ell}}}\\
&\times \cos(\gamma^- \log p_1) \cdots \cos(\gamma^-\log p_j) \cos(\gamma^- \log q_1 )\cdots \cos(\gamma^- \log q_{\ell}) \\
&+ O( (|\gamma^+|+T^{0.1}) T^{0.1} (\log \log T)^{2k})\\
&= T \prod_{1\le i \le \I} \sum_{0\le m \le 200 \beta_i^{-3/4} } k^m \sum_{\substack{ j+\ell=m \\ 0\le j,\ell \le 100 \beta_i^{-3/4} }}\frac{1}{j!\ell!} \sum_{T^{\beta_{i-1}}< p_1,\ldots,p_m \le T^{\beta_{i}}} \frac{g(p_1\cdots p_m)}{\sqrt{p_1\cdots p_m}}\\
&\times \cos(\gamma^- \log p_1) \cdots \cos(\gamma^- \log p_m) 
+O( (|\gamma^+|+T^{0.1}) T^{0.1} (\log \log T)^{2k})\\
&\le T \prod_{1\le i \le \I} \sum_{0\le m \le 200 \beta_i^{-3/4} } \frac{k^m 2^m}{m!}  \sum_{T^{\beta_{i-1}}< p_1,\ldots,p_m \le T^{\beta_{i}}} \frac{g(p_1\cdots p_m)}{\sqrt{p_1\cdots p_m}} \cos(\gamma^- \log p_1) \cdots \cos(\gamma^- \log p_m) \\
&+O( (|\gamma^+|+T^{0.1}) T^{0.1} (\log \log T)^{2k}),
\end{align*}
where the last inequality makes use of the non-negativity of the inner summand. 
Since $g$ is supported on squares, we must have that $m$ is even, say $m=2n$ with $n \ge 0$.  By relabelling the prime variables as 
$q_1, \ldots , q_{2n}$, we see that 
\begin{equation}
\label{fancyF1}
\begin{split}
 \mathscr{I} & \ll T \prod_{1\le i \le \I} \sum_{0\le n \le 100 \beta_i^{-3/4} } \frac{k^{2n} 2^{2n}}{(2n)!}  \sum_{T^{\beta_{i-1}}< q_1,\ldots,q_{2n} \le T^{\beta_{i}}} 
 \frac{g(q_1\cdots q_{2n})}{\sqrt{q_1\cdots q_{2n}}} \cos(\gamma^- \log q_1) \cdots \cos(\gamma^- \log q_{2n}) \\
&+O( (|\gamma^+|+T^{0.1}) T^{0.1} (\log \log T)^{2k}).
\end{split}
\end{equation}
Next, we observe that $q_1 \cdots q_{2n}$ is a square if and only it equals $p_{1}^2 \cdots p_{n}^2$ for some primes $p_u \in [T^{\beta{i-1}}, T^{\beta_i}]$
with $1 \le u \le n$.  
Grouping terms according to $q_1 \cdots q_{2n} =p_{1}^2 \cdots p_{n}^2  $ gives 
\begin{align*}
&\sum_{T^{\beta_{i-1}}< q_1,\ldots,q_{2n} \le T^{\beta_{i}}} 
 \frac{g(q_1\cdots q_{2n})}{\sqrt{q_1\cdots q_{2n}}} \cos(\gamma^- \log q_1) \cdots \cos(\gamma^- \log q_{2n})\\
 &= \sum_{T^{\beta_{i-1}}< p_1,\ldots,p_{n} \le T^{\beta_{i}}} 
  \sum_{\substack{T^{\beta_{i-1}}< q_1,\ldots,q_{2n} \le T^{\beta_{i}} \\ q_1\cdots q_{2n} = (p_1\cdots p_{n})^2}}
 \frac{g(p_1^2\cdots p_{n}^2)}{\sqrt{p_1^2\cdots p_{n}^2}} \cos^2(\gamma^- \log p_1) \cdots \cos^2(\gamma^- \log p_{n})
 \\
 &\times {\#\{  (p_1',\ldots, p_{n}')\mid  p_1'\cdots p_{n}'=p_1\cdots p_{n} \}}^{-1}\\
 &= \sum_{T^{\beta_{i-1}}< p_1,\ldots,p_{n} \le T^{\beta_{i}}} 
 \frac{g(p_1^2\cdots p_{n}^2)}{\sqrt{p_1^2\cdots p_{n}^2}} \cos^2(\gamma^- \log p_1) \cdots \cos^2(\gamma^- \log p_{n})\\
&\times  \frac{ {\#\{  (q_1,\ldots, q_{2n})\mid  q_1\cdots q_{2n}=(p_1 \cdots p_{n})^2 \}}  }{ {\#\{  (p_1', \ldots , p_{n}')\mid  p_1'\cdots p_{n}'=p_1\cdots p_{n} \}} }.
\end{align*}
In the above the factor ${\#\{  (p_1',\ldots, p_{n}')\mid  p_1'\cdots p_{n}'=p_1\cdots p_{n} \}}^{-1}$ accounts for possible repetitions when counting squares $p_1^2 \cdots p_n^2$.
With this observation, by following the argument in  \cite[p. 11]{Ha}, we have that the first term in \eqref{fancyF1} equals
\begin{align} \label{Ha-lemma1-mid3}
 \begin{split}
& T \prod_{1\le i \le \I} \sum_{0\le n \le 100 \beta_i^{-3/4}} \frac{(2k)^{2n}}{(2n)!} \sum_{T^{\beta_{i-1}}< p_1,\ldots,p_n \le T^{\beta_{i}}} \frac{g(p^2_1\cdots p^2_n)}{p_1\cdots p_n} \frac{\#\{(q_1\ldots q_{2n})\mid  q_1\cdots q_{2n}=p_1^2\cdots p^2_{n} \}}{\#\{  (q_1\ldots q_{n})\mid  q_1\cdots q_{n}=p_1\cdots p_{n} \}}\\
&\times\cos^2(\gamma^- \log p_1) \cdots \cos^2(\gamma^- \log p_n),
 \end{split}
\end{align}
where each $q_i$ again denotes a prime in $(T^{\beta_{i-1}},T^{\beta_{i}}]$.

By \cite[Eq. (4.2)]{Ha}, we know 
$$
g(p^2_1\cdots p^2_n) =\frac{1}{2^{2n}} \prod_{j=1}^r\frac{ (2\alpha_j)!}{(\alpha_j!)^2}
$$
and
$$
\frac{\#\{(q_1\ldots q_{2n})\mid  q_1\cdots q_{2n}=p_1^2\cdots p^2_{n} \}}{\#\{  (q_1\ldots q_{n})\mid  q_1\cdots q_{n}=p_1\cdots p_{n} \}}
=\frac{(2n)!}{\prod_{j=1}^r(2\alpha_j)! } \left(\frac{n!}{\prod_{j=1}^r\alpha_j! } \right)^{-1}
$$
whenever $p_1\cdots p_{n}$ is a product of $r$ distinct primes with multiplicities $\alpha_1,\ldots,\alpha_r$ (in particular, $\alpha_1+\cdots+\alpha_r=n$).  Therefore, the expression \eqref{Ha-lemma1-mid3} is equal to 
\begin{align*}
& T \prod_{1\le i \le \I} \sum_{0\le n \le 100 \beta_i^{-3/4}} \frac{k^{2n}}{n!} \sum_{T^{\beta_{i-1}}< p_1,\ldots,p_n \le T^{\beta_{i}}} \frac{\cos^2(\gamma^- \log p_1) \cdots \cos^2(\gamma^- \log p_n)}{p_1\cdots p_n} \frac{1}{\prod_{j=1}^r \alpha_j {!}}\\
&\le T \prod_{1\le i \le \I} \sum_{0\le n \le 100 \beta_i^{-3/4}} \frac{1}{n!} \left(k^2 \sum_{T^{\beta_{i-1}}< {p} \le T^{\beta_{i}}} \frac{\cos^2(\gamma^- \log p)}{p} \right)^n\\
&\le T\exp\left( k^2 \sum_{p\le T^{\beta_{\I}}} \frac{\cos^2(\gamma^- \log p)}{p} \right).
\end{align*}
Hence, we arrive at
\begin{equation}
  \label{Isecondlast}
\mathscr{I} \ll T\exp\left( k^2 \sum_{p\le T^{\beta_{\I}}} \frac{\cos^2(\frac{1}{2}(\alpha_1-\alpha_2) \log p)}{p} \right) + (|\gamma^+|+T^{0.1}) T^{0.1} (\log \log T)^{2k}.
\end{equation}
Since $\beta_{\I} < 1$ and $\cos^2(\theta) = \frac{1}{2}(1+\cos(2 \theta))$, from \eqref{Mertens-variant}, it follows that 
\begin{equation}
\begin{split}
  \label{cosbd}
\sum_{p\le T^{\beta_{\I}}} \frac{\cos^2( \frac{1}{2}(\alpha_1-\alpha_2) \log p)}{p}
&\le \sum_{p\le T} \frac{\cos^2( \frac{1}{2}(\alpha_1-\alpha_2) \log p)}{p} \\
&  =\frac{1}{2}\sum_{p\le T} \frac{1}{p}  +\frac{1}{2}\sum_{p\le T} \frac{\cos((\alpha_1-\alpha_2) \log p)}{p}\\
& \le \frac{1}{2} \log\log T  +\frac{1}{2}\log(\mathcal{F}(T,\alpha_1,\alpha_2)) +O(1),
\end{split}
\end{equation}
where $\mathcal{F}(T,\alpha_1,\alpha_2)$ is defined in  \eqref{def-mathcalF}.
Therefore, by \eqref{Isecondlast}, \eqref{cosbd}, and \eqref{a1plusa2cond}, 
\begin{equation*}
  \mathscr{I} \ll_k T (\log T)^{\frac{k^2}{2}}\mathcal{F}(T,\alpha_1,\alpha_2)^{\frac{k^2}{2}}+ T^{0.8} (\log \log T)^{2k}
   \ll T (\log T)^{\frac{k^2}{2}}\mathcal{F}(T,\alpha_1,\alpha_2)^{\frac{k^2}{2}}.
\end{equation*}
This combined with \eqref{Ha-lemma1-start} and \eqref{Ha-lemma1-startB} completes the proof of Lemma \ref{refind_Harper_Lemma1}.

\section{Proof of Lemma \ref{refind_Harper_Lemma2}}

Following Harper, by the definition $\mathcal{S}(j)$, we can bound the left hand side of \eqref{refind_Harper_est2} by
\begin{equation}\label{lemma2-eq-1}
\sum_{\ell=j+1}^{\mathcal{I}} \int_{T}^{2T}  \left( \prod_{1\le i \le j} (\exp (k \Re G_{i, j}(t))^2 \right) \mathds{1}_{A_{j,\ell}}(t)  dt,\footnote{For $S \subset \mathbb{R}$, $\mathds{1}_S(t)=1$ if $ t \in S$ and $\mathds{1}_S(t)=0$ if $ t \notin S$ (the indicator function of $S$).}
\end{equation}
where 
\begin{equation*}
   A_{j,\ell} := \Big\{ t \in \mathbb{R} \ | \ |\Re G_{i,j}(t)|\le \beta_i^{-3/4}, \,  \forall 1 \le i \le j, \text{ but } |\Re G_{j+1,\ell}(t)|> \beta_{j+1}^{-3/4} \Big\}.
\end{equation*}
From the definition of $A_{j,\ell}$, it follows that 
\begin{equation*}
    \mathds{1}_{A_{j,\ell}}(t)  \le (\beta_{j+1}^{3/4} |\Re G_{j+1,\ell}(t) |)^{M}
\end{equation*}
for any positive integer $M$.  From this point on, we set 
\begin{equation*}
  \label{M}
  M = 2[1/(10\beta_{j+1})].
\end{equation*}
Therefore the integral in \eqref{lemma2-eq-1} is
\begin{align}\label{lemma2-eq-2}
 \begin{split}
&\le  \int_{T}^{2T} \prod_{1\le i \le j} (\exp (k \Re G_{i, j}(t))^2 (\beta_{j+1}^{3/4} |\Re G_{j+1,\ell}(t) |)^{M} dt\\
&\ll (\beta_{j+1}^{3/2})^{[1/(10\beta_{j+1})]}  \int_{T}^{2T} \prod_{1\le i \le j} \left( \sum_{0\le n \le 100k \beta_{i}^{-3/4}} \frac{(k\Re G_{i,j}(t))^n}{n!} \right)^2  (\Re G_{j+1,\ell}(t) )^{M} dt,
 \end{split}
\end{align}
The second estimate can be established similar to the proof of Lemma 5.2 of \cite{Ki}. 
Arguing as in the proof of Lemma \ref{refind_Harper_Lemma1} and \cite[p. 13]{Ha}, since $g(p_1\cdots p_m)\cos(\gamma^- \log p_1) \cdots \cos(\gamma^- \log p_m) \ge 0$, we deduce 
\begin{align}\label{lemma2-eq-3}
 \begin{split}
& \int_{T}^{2T} \prod_{1\le i \le j} \left( \sum_{0\le n \le 100k \beta_{i}^{-3/4}} \frac{(k\Re G_{i,j}(t))^n}{n!} \right)^2  (\Re G_{j+1,\ell}(t))^{M} dt\\
&\ll  T \prod_{1\le i \le j} \sum_{0\le m \le 200k \beta_i^{-3/4} } \frac{k^m 2^m}{m!}  \sum_{T^{\beta_{i-1}}< p_1,\ldots,p_m \le T^{\beta_{i}}} \frac{g(p_1\cdots p_m)}{\sqrt{p_1\cdots p_m}} \cos(\gamma^- \log p_1) \cdots \cos(\gamma^- \log p_m) \\
&\times \sum_{T^{\beta_{j}}< p_1,\ldots,p_M \le T^{\beta_{j+1}}} \frac{g(p_1\cdots p_M)}{\sqrt{p_1\cdots p_M}} \cos(\gamma^- \log p_1) \cdots \cos(\gamma^- \log p_M)\\
&+O( (|\gamma^+|+T^{0.3}) T^{0.3}  (\log \log T)^{2k}).
\end{split}
\end{align}
Since $g$ is supported on squares, by following an argument similar to the proof of Lemma \ref{refind_Harper_Lemma1}, we find that 
the previous expression is bounded by 
\begin{align}\label{lemma2-eq-4}
\begin{split}
&\ll T \exp\left( k^2 \sum_{p\le T^{\beta_{j}}} \frac{\cos^2(\gamma^- \log p)}{p} \right) \times \frac{M!}{2^M (M/2)!} \left( \sum_{T^{\beta_j} \le p\le T^{\beta_{j+1}}} \frac{\cos^2(\gamma^- \log p)}{p}  \right)^{M/2}
\\
&+O(  (|\gamma^+|+T^{0.3}) T^{0.3} (\log \log T)^{2k})\\
&\ll T \exp\left( k^2 \sum_{p\le T^{\beta_{j}}} \frac{\cos^2(\gamma^- \log p)}{p} \right) \times \left( \frac{1}{20 \beta_{j+1}} \sum_{T^{\beta_j} \le p\le T^{\beta_{j+1}}} \frac{\cos^2(\gamma^- \log p)}{p} \right)^{[1/(10\beta_{j+1})]}
\\
&+O(  (|\gamma^+|+T^{0.3}) T^{0.3} (\log \log T)^{2k}).
 \end{split}
\end{align}
Hence, by \eqref{lemma2-eq-1}, \eqref{lemma2-eq-2}, \eqref{lemma2-eq-3}, and \eqref{lemma2-eq-4}, we arrive at
\begin{align}\label{lemma2-final}
 \begin{split}
&\int_{t\in \mathcal{S}(j)} \exp \left( 2k \Re\sum_{ p\le T^{\beta_{j}}}\frac{\cos( \frac{1}{2}(\alpha_1-\alpha_2) \log p)}{p^{\frac{1}{2}+\frac{1}{\beta_{j} \log T}+\mi (t+\frac{1}{2}(\alpha_1+\alpha_2)) } }\frac{\log(T^{\beta_{j}}/p)}{\log T^{\beta_{j}} }  \right)dt \\
&\ll_k (\mathcal{I}-j) T \exp\left( k^2 \sum_{p\le T^{\beta_{j}}} \frac{\cos^2( \frac{1}{2}(\alpha_1-\alpha_2)  \log p)}{p} \right) \\
& \times \left( \frac{\beta_{j+1}^{1/2}}{20} \sum_{T^{\beta_j} < p\le T^{\beta_{j+1}}} \frac{\cos^2(  \frac{1}{2}(\alpha_1-\alpha_2)  \log p)}{p} \right)^{[1/(10\beta_{j+1})]}\\
&+(\mathcal{I}-j) ( (|\gamma^+|+T^{0.3}) T^{0.3}) (\log \log T)^{2k}.
 \end{split}
\end{align}
Recall that $\mathcal{I}\le \log\log\log T$, $\beta_0=0$,  $\beta_1=\frac{1}{(\log\log T)^2}$, and 
$$
\sum_{p\le T^{\beta_1}} \frac{1}{p} \le \log\log T.
$$
Observe that  for $j=0$, the left of \eqref{lemma2-final} is $\meas(\mathcal{S}(0))$. Therefore, by the trivial bound $\cos^2(\frac{1}{2}(\alpha_1-\alpha_2) \log p)\le 1$ and  the assumption $|\gamma^+|=|\frac{\alpha_1+\alpha_2}{2}|\ll T^{0.6}$,  we derive $\meas(\mathcal{S}(0)) \ll Te^{-(\log\log T)^2/10}$. 

For $1\le j\le \mathcal{I}-1$, we have $\mathcal{I}-j \le \frac{\log(1/\beta_j)}{\log 20}$ and
$$
\sum_{T^{\beta_j} < p\le T^{\beta_{j+1}}} \frac{\cos^2( \frac{1}{2}(\alpha_1-\alpha_2)  \log p)}{p}  \le \sum_{T^{\beta_j} < p\le T^{\beta_{j+1}}} \frac{1}{p} = \log \beta_{j+1} -\log \beta_{j} +o(1) \le 10.
$$
Thus, by \eqref{Mertens-variant} and  the assumption $|\gamma^+|\ll T^{0.6}$, we see that  the left of \eqref{lemma2-final} is
$$
\ll_k T (\log T)^{\frac{k^2}{2}} \mathcal{F}(T,\alpha_1,\alpha_2)^{\frac{k^2}{2}} \exp \left(- \frac{\log(1/\beta_{j+1})}{21 \beta_{j+1}}\right)
$$
as desired.

\section{Proof of Lemma \ref{refind_Harper_Lemma3}}

The proof of Lemma \ref{refind_Harper_Lemma3} is similar to the proofs of Lemmas  \ref{refind_Harper_Lemma1} and 
 \ref{refind_Harper_Lemma2} 
 One key difference is that we need to invoke Lemma \ref{refined-H-R} in place of Lemma \ref{formula-f}.
 In this section, we shall  establish the estimate \eqref{replace}. As the  proof of \eqref{replace2} is similar, the 
details shall be omitted.  The integral in \eqref{replace} shall be denoted $\int_{\mathcal{T}} \exp(\varphi(t)) \, dt$ where $\exp(\varphi(t)) $ is the  integrand in \eqref{replace}.  First, we decompose this integrand in terms of integer parameters $m$ 
satisfying $0 \le m \le \frac{\log\log T}{\log 2}$.  For each such $m$, we define 
\begin{equation} 
 \label{Pmt}
 P_m(t) = \sum_{2^m< p\le 2^{m+1}} \frac{\cos( (\alpha_1-\alpha_2)\log p)}{2p^{1+ \mi (2t+ (\alpha_1+\alpha_2) )}}.
 \end{equation}
 and the sets 
\begin{equation}
\begin{split}
  \label{Pm}
&\mathcal{P}(m)\\
& := \left\{ t\in [T, 2T] \mid  |\Re P_m(t)|> 2^{-m/10},  \text{ but } |\Re P_n(t)|\le 2^{-n/10} \text{ for every } m+1\le n \le \frac{\log\log T}{\log 2}\right\}.
\end{split}
\end{equation}
 Observe that we have the identity
 \begin{align}
 \label{idenP}
     \sum_{p\le \log T }\frac{\cos( (\alpha_1-\alpha_2) \log p)}{2p^{1+ \mi (2t+ (\alpha_1+\alpha_2) )}}
   = \sum_{0 \le m \le  \frac{\log \log T}{\log 2}}  P_m(t). 
\end{align}
We now have the decomposition
\begin{equation}
\label{twoterms}
  \int_{\mathcal{T}} \exp( \varphi(t))) \, dt = \sum_{0 \le m \le \frac{\log \log T}{\log 2}} \int_{\mathcal{T} \cap \mathcal{P}(m)} 
  \exp( \varphi(t))) \, dt  + \int_{\mathcal{T} \cap \left( \bigcap_{m} \mathcal{P}(m)^c  \right)}   \exp( \varphi(t))) \, dt.
\end{equation}
In order to establish \eqref{replace}, we shall bound each of the integrals on the right side of \eqref{twoterms}.

If $t$ does not belong to any $\mathcal{P}(m)$, then $|\Re P_n(t)|\le 2^{-n/10}$ for all $n\le \frac{\log\log T}{\log 2}$. (Indeed, for those $t$ belonging to none of $\mathcal{P}(m)$, $0\le m\le \frac{\log\log T}{\log 2}$, if $|\Re P_m(t)|> 2^{-m/10}$ for some $0 \leq m \le \frac{\log\log T}{\log 2}$, then $|\Re P_L(t)|> 2^{-L/10}$ for some $m+1\le L\le \frac{\log\log T}{\log 2}$ as $t\notin \mathcal{P}(m)$. Choosing $L$ to be maximal, we then have $|\Re P_n(t)|\le 2^{-n/10}$ for every $L+1\le n \le \frac{\log\log T}{\log 2}$, which means $t\in \mathcal{P}(L)$, a contradiction.) For such an instance,  $\Re \sum_{p\le \log T }\frac{\cos( (\alpha_1-\alpha_2)\log p)}{2p^{1+ \mi (2t+ (\alpha_1+\alpha_2) )}} =O(1)$.  Hence, the contribution of such $t$ to the integral $\int_{\mathcal{T}}$ can be bounded by using Lemma \ref{refind_Harper_Lemma1}. 
That is, 
\begin{equation}
\begin{split}
  \label{firstbound}
  &  \int_{\mathcal{T} \cap \left( \bigcap_{m} \mathcal{P}(m)^c  \right)}   \exp( \varphi(t))) \, dt  \\
   & \ll  \int_{\mathcal{T} \cap \left( \bigcap_{m} \mathcal{P}(m)^c  \right)}  \exp \left( 2k \Re\sum_{ p\le T^{\beta_{\I}}}\frac{\cos(\frac{1}{2}(\alpha_1-\alpha_2)\log p)}{p^{\frac{1}{2}+\frac{1}{\beta_{\I} \log T}+\mi (t+\frac{1}{2}(\alpha_1+\alpha_2)) } }\frac{\log(T^{\beta_{\I}}/p)}{\log T^{\beta_{\I}} }  \right)dt \\
&\ll_k T (\log T)^{\frac{k^2}{2}} \left(\mathcal{F}(T,\alpha_1,\alpha_2) \right)^{\frac{k^2}{2}}.
\end{split}
  \end{equation}

It remains to estimate the contribution from $t\in \mathcal{T} \cap \mathcal{P}(m)$, with  $0 \le m \le \frac{\log\log T}{\log 2}$, to $\int_{\mathcal{T}}$ (more precisely, the first  integral on the right of \eqref{twoterms}).  To do so, we first consider the case that $0\le m\le \frac{2\log\log\log T}{\log 2}$. In this case, we have
  \begin{align*}
 \left|\Re  \sum_{p\le \log T }\frac{\cos( (\alpha_1-\alpha_2) \log p)}{2p^{1+ \mi (2t+ (\alpha_1+\alpha_2) )}} \right|
 &\leq  \sum_{0\leq n \leq \frac{\log \log T}{\log 2}} | \Re P_n(t)|\\
 &\leq  \sum_{0\leq n  \leq m} |\Re P_n(t) | +   \sum_{m+1 \leq  n  \leq \frac{\log \log T}{\log 2}} | \Re P_n(t) |\\
  &\leq   \sum_{ p\le 2^{m+1}} \frac{1}{2p}  +  \sum_{m + 1 \leq n  \leq \frac{\log \log T}{\log 2}}  \frac{1}{2^{n/10}}.
  \end{align*}
  The last inequality makes use of the definition of $ P_m(t) $ in \eqref{Pm}. Therefore,
\begin{align*}
&\left| \Re\left( \sum_{ p\le 2^{m+1}}\frac{\cos( \frac{1}{2}(\alpha_1-\alpha_2) \log p)}{p^{\frac{1}{2}+\frac{1}{\beta_{\I} \log T}+\mi (t+ \frac{1}{2}(\alpha_1+\alpha_2) ) } }\frac{\log(T^{\beta_{\I}}/p)}{\log T^{\beta_{\I}} }  + \sum_{p\le \log T }\frac{\cos( (\alpha_1-\alpha_2) \log p)}{2p^{1+ \mi (2t+ (\alpha_1+\alpha_2) )}} \right) \right| \\
&\le \sum_{ p\le 2^{m+1}}\frac{1}{\sqrt{p}} + \sum_{ p\le 2^{m+1}}\frac{1}{2p} 
 +\sum_{m+1\le n \le \frac{\log\log T}{\log 2}}  \frac{1}{2^{n/10}}\\
 &\ll 2^{m/2}.
\end{align*}
Thus, we derive
\begin{align}\label{Ha-lemma3-mid1}
 \begin{split}
& \int_{t\in \mathcal{T}\cap \mathcal{P}(m)} \exp \left( 2k \Re \left( \sum_{ p\le T^{\beta_{\I}}}\frac{\cos( \frac{1}{2}(\alpha_1-\alpha_2) \log p)}{p^{\frac{1}{2}+\frac{1}{\beta_{\I} \log T}+\mi (t+ \frac{1}{2}(\alpha_1+\alpha_2) ) } }\frac{\log(T^{\beta_{\I}}/p)}{\log T^{\beta_{\I}} }  +  \sum_{p\le \log T }\frac{\cos( (\alpha_1-\alpha_2) \log p)}{2p^{1+ \mi (2t+ (\alpha_1+\alpha_2) )}}  \right) \right)dt \\
&\le e^{O(k2^{m/2})} \int_{t\in \mathcal{T}\cap \mathcal{P}(m)} \exp \left( 2k \Re\sum_{2^{m+1}< p\le T^{\beta_{\I}}}\frac{\cos(\frac{1}{2}(\alpha_1-\alpha_2)\log p)}{p^{\frac{1}{2}+\frac{1}{\beta_{\I} \log T}+\mi (t+\frac{1}{2}(\alpha_1+\alpha_2)) } }\frac{\log(T^{\beta_{\I}}/p)}{\log T^{\beta_{\I}} }  \right)dt \\
&\le e^{O(k2^{m/2})} \int_{t\in \mathcal{T}}  \exp \left( 2k \Re\sum_{2^{m+1}< p\le T^{\beta_{\I}}}\frac{\cos(\frac{1}{2}(\alpha_1-\alpha_2)\log p)}{p^{\frac{1}{2}+\frac{1}{\beta_{\I} \log T}+\mi (t+\frac{1}{2}(\alpha_1+\alpha_2)) } }\frac{\log(T^{\beta_{\I}}/p)}{\log T^{\beta_{\I}} }  \right) \\  &\times (2^{m/10} \Re P_m(t))^{2[2^{3m/4}]}dt. 
 \end{split}
\end{align}
Let $N =2[2^{3m/4}] $. Arguing as in the proof of Lemma \ref{refind_Harper_Lemma1} while applying Lemma \ref{refined-H-R} instead of Lemma \ref{formula-f}, we see the above integral is
\begin{align}\label{Ha-lemma3-mid2}
 \begin{split}
&\ll   T \exp\left( k^2 \sum_{2^{m+1} < p\le T^{\beta_{\mathcal{I}}}} \frac{\cos^2(\frac{1}{2}(\alpha_1-\alpha_2) \log p)}{p} \right)\times 2^{mN/10}  \sum_{2^m< p_1,\ldots,p_{N} \le 2^{m+1}} \frac{g(p_1\cdots p_{N})}{p_1\cdots p_{N}} \\
&+ (|\gamma^+| +T^{0.1+o(1)}) T^{0.1} (\log\log T)^{2k} \\
&\ll T \exp\left( k^2 \sum_{2^{m+1} <p\le T^{\beta_{\mathcal{I}}}} \frac{\cos^2(\frac{1}{2}(\alpha_1-\alpha_2) \log p)}{p} \right)\times  \frac{2^{mN/10} N!}{2^{N} (N/2)!} \left( \sum_{2^m < p\le 2^{m+1}} \frac{1}{p^2}  \right)^{N/2} \\
&+T^{0.7} (\log\log T)^{2k}
 \end{split}
\end{align}
 as  $|\gamma^+|= |\frac{\alpha_1+\alpha_2}{2}|\ll T^{0.6}$.
Indeed, recall that $2^{m}\le \log T$, and note that if $2^m< p_1,\ldots,p_{N} \le 2^{m+1}$, then
$$
 p_1^2\cdots p_{N}^2 \le 2^{2(m+1)N}  \le 2^{6 (\log\log T) (\log T)^{3/4}}.
$$
Therefore, the additional contribution of $(\Re P_m(t))^{N}$ enlarges the last big-O term \eqref{Ha-lemma1-mid2} as
$$
(|\gamma^+| + T^{0.1 }2^{6 (\log\log T) (\log T)^{3/4}} )  \sum_{2^m< p_1,\ldots,p_{N} \le 2^{m+1}} \frac{1}{p_1\cdots p_{N}} T^{0.1}(\log\log T)^{2k},
$$
where the sum is equal to
$$
\left( \sum_{2^m < p\le 2^{m+1}} \frac{1}{p} \right)^{N} \le \left( (2^{m+1} -2^m)  \frac{1}{2^m} \right)^{N} =1.
$$
Now, by \eqref{Ha-lemma3-mid2}, the left of \eqref{Ha-lemma3-mid1} is
\begin{align}
\label{cas2}
\begin{split}
&\ll e^{O(k2^{m/2})} \left(2^{m/5}\cdot 2^{3m/4}\cdot 2^{-m}  \right)^{[2^{3m/4}]}  T  (\log T)^{\frac{k^2}{2}} \mathcal{F}(T,\alpha_1,\alpha_2)^{\frac{k^2}{2}} \\
&\ll  e^{O(k2^{m/2}) -2^{3m/4}}  T  (\log T)^{\frac{k^2}{2}} \mathcal{F}(T,\alpha_1,\alpha_2)^{\frac{k^2}{2}}. 
\end{split}
\end{align}

Secondly, we evaluate the contribution from $t\in \mathcal{T} \cap \mathcal{P}(m)$ with $ \frac{2\log\log\log T}{\log 2} <  m \le \frac{\log\log T}{\log 2}$. We shall consider 
\begin{align}
 \begin{split}
 \int_{t\in \mathcal{T}\cap \mathcal{P}(m)}1\, dt \leq
 \int_{t\in \mathcal{T}}  (2^{m/10} \Re P_m(t))^{2[2^{3m/4}]}\, dt.
 \end{split}
\end{align}
Following the previous  argument in  \eqref{Ha-lemma3-mid1}  where the exponential factor is replaced by $1$,  one can show that $\meas(\mathcal{T}\cap \mathcal{P}(m))\ll T e^{ -2^{3m/4}}$. So, for $2^m\ge (\log\log T)^2$,  we see 
$\meas(\mathcal{T}\cap\mathcal{P}(m))\ll T e^{ -(\log \log T)^{3/2}}$.
In addition, the Cauchy-Schwarz inequality tells us that 
\begin{align}
\label{Pmleft}
\begin{split}
&\int_{t\in \mathcal{T}\cap \mathcal{P}(m)} \exp \left( 2k \Re \left( \sum_{ p\le T^{\beta_{\I}}}\frac{\cos( \frac{1}{2}(\alpha_1-\alpha_2) \log p)}{p^{\frac{1}{2}+\frac{1}{\beta_{\I} \log T}+\mi (t+ \frac{1}{2}(\alpha_1+\alpha_2) ) } }\frac{\log(T^{\beta_{\I}}/p)}{\log T^{\beta_{\I}} }  +  \sum_{p\le \log T }\frac{\cos( (\alpha_1-\alpha_2) \log p)}{2p^{1+ \mi (2t+ (\alpha_1+\alpha_2) )}}  \right) \right)dt\\
&\ll e^{k \log \log \log T } \int_{t\in \mathcal{T}\cap \mathcal{P}(m)} \exp \left( 2k \Re \sum_{ p\le T^{\beta_{\I}}}\frac{\cos( \frac{1}{2}(\alpha_1-\alpha_2) \log p)}{p^{\frac{1}{2}+\frac{1}{\beta_{\I} \log T}+\mi (t+ \frac{1}{2}(\alpha_1+\alpha_2) ) } }\frac{\log(T^{\beta_{\I}}/p)}{\log T^{\beta_{\I}} }    \right)dt \\
&\ll (\log \log T)^k
\left(  \int_{t\in \mathcal{T}\cap \mathcal{P}(m)} \exp \left( 4k \Re\sum_{ p\le T^{\beta_{\I}}}\frac{\cos( \frac{1}{2}(\alpha_1-\alpha_2) \log p)}{p^{\frac{1}{2}+\frac{1}{\beta_{\I} \log T}+\mi (t+ \frac{1}{2}(\alpha_1+\alpha_2) ) } }\frac{\log(T^{\beta_{\I}}/p)}{\log T^{\beta_{\I}} }   \right) dt \right)^{\frac{1}{2}}\\
& \times  (\meas(\mathcal{T}\cap \mathcal{P}(m)))^{\frac{1}{2}}.
\end{split}
\end{align}
As Lemma \ref{refind_Harper_Lemma1} gives
\begin{align*}
 \int_{t\in \mathcal{T}\cap \mathcal{P}(m)} \exp \left( 4k \Re  \sum_{ p\le T^{\beta_{\I}}}\frac{\cos( \frac{1}{2}(\alpha_1-\alpha_2) \log p)}{p^{\frac{1}{2}+\frac{1}{\beta_{\I} \log T}+\mi (t+ \frac{1}{2}(\alpha_1+\alpha_2) ) } }\frac{\log(T^{\beta_{\I}}/p)}{\log T^{\beta_{\I}} }  \right)dt
 \ll T (\log T)^{4k^2},
\end{align*}
we see that \eqref{Pmleft} is bounded by 
\begin{align} 
\label{cas3}
 \ll_k T e^{ -\frac{1}{4}(\log \log T)^{3/2}} .
\end{align} 

Finally, we complete the proof of the Lemma by combining \eqref{twoterms}  with the bounds \eqref{cas2}
($0\le m\le \frac{2\log\log\log T}{\log 2}$) and \eqref{cas3} ($ \frac{2\log\log\log T}{\log 2} < m \leq \frac{\log\log T}{\log 2} $ )
for $\int_{\mathcal{T} \cap \mathcal{P}(m)} 
  \exp( \varphi(t))) \, dt$ and the bound \eqref{firstbound}.

\end{document}